\documentclass[reqno,11pt]{amsart}
\usepackage{amsmath,amsfonts,amsthm,amssymb,mathrsfs}
\usepackage{setspace,comment,verbatim}
\usepackage{fancyhdr}
\usepackage{lastpage}
\usepackage{extramarks}
\usepackage{chngpage}
\usepackage{soul,color}
\usepackage{graphicx,float,wrapfig}
\DeclareMathSymbol{\R}{\mathbin}{AMSb}{"52}
\newcommand{\A}{\mathcal{A}}

\newcommand{\E}{\mathbb{E}}
\newcommand{\F}{\mathscr{F}}
\renewcommand{\H}{\mathscr{H}}

\renewcommand{\R}{\mathbb{R}}
\newcommand{\inv}{\text{inv}}

\newcommand{\cross}{\text{cross}}
\newcommand{\nest}{\text{nest}}
\newcommand{\Cross}{\text{Cross}}
\newcommand{\Nest}{\text{Nest}}

\DeclareMathSymbol{\N}{\mathbin}{AMSb}{"4E}
\DeclareMathSymbol{\Z}{\mathbin}{AMSb}{"5A}

\newtheorem{thm}{Theorem}
\newtheorem*{thm*}{Theorem}
\newtheorem{lem}{Lemma}
\newtheorem{remark}{Remark}
\newtheorem{defn}{Definition}

\newtheorem{cor}{Corollary}

\newtheorem{assump}{Condition}
\begin{document}
\title{Two-parameter Non-commutative Central Limit Theorem}
\author{Natasha Blitvi\'c\\}

\begin{abstract}
The\thanks{\footnotesize This research was supported by the \emph{Claude E. Shannon Research Assistantship} at MIT and the \emph{Chateaubriand Fellowship}, held at the Institut Gaspard Monge of the Univerist\'e Paris-Est.\\} non-commutative Central Limit Theorem (CLT) introduced by Speicher in 1992 states that given a sequence $(b_i)_{i\geq 1}$ of non-commutative random variables satisfying the commutation relation
$$b_i^{\epsilon}b_j^{\epsilon'}=s(j,i)\,b_j^{\epsilon'}b_i^{\epsilon},\quad\quad s(j,i)\in\{-1,1\},$$
the $\ast$-moments of the normalized partial sum $S_N=(b_1+\ldots + b_N)/\sqrt N$ are given by a Wick-type formula refined to count the number of crossings in the underlying pair-partitions. When coupled with explicit matrix models, the theorem yields random matrix models for creation and annihilation operators on the $q$-Fock space of Bo\.zejko and Speicher.

In this paper, we derive a non-commutative CLT when the commutation relation is generalized to $$b_i^{\epsilon}b_j^{\epsilon'}=\mu_{\epsilon',\epsilon}(j,i)\,b_j^{\epsilon'}b_i^{\epsilon}, \quad \quad \mu_{\epsilon',\epsilon}(j,i)\in\mathbb R.$$ 
The statistics of the limiting random variable are a second-parameter refinement of those above, jointly indexing the number of crossings \emph{and nestings} in the underlying pair-partitions. Coupled with analogous matrix constructions, the theorem yields random matrix models for creation and annihilation operators on the recently introduced $(q,t)$-Fock space.
\end{abstract}
\maketitle

\section{Introduction}
\label{main}
In \emph{non-commutative probability}, probabilistic interpretations of operator algebraic frameworks give rise to non-commutative analogues of classical results in probability theory. The general setting is that of a non-commutative probability space $(\mathcal A,\varphi)$, formed by a $\ast$-algebra $\mathcal A$, containing the \emph{non-commutative random variables}, and a positive linear functional $\varphi:\mathcal A\to \mathbb C$, playing the role of \emph{expectation}. A particularly rich non-commutative probabilistic framework is Voiculescu's free probability \cite{Voiculescu1986}, which has been found to both parallel and complement the classical theory (see in-depth treatments in \cite{Biane2003,NicaSpeicher,Voiculescu1992}). 
Whereas free probability can be seen as characterized by the absence of commutative structure, a parallel -- albeit somewhat slower -- development has targeted non-commutative settings built around certain types of \emph{commutation relations}. 

In \cite{Speicher1992}, Speicher showed a non-commutative version of the classical Central Limit Theorem (CLT) for mixtures of commuting and anti-commuting elements.
Speicher's CLT concerns a sequence of elements $b_1,b_2\ldots\in\A$ whose terms pair-wise satisfy the deformed commutation relation $b_ib_j=s(j,i)b_jb_i$ with $s(j,i)\in\{-1,1\}$. It is not a priori clear that the partial sums \begin{equation}S_N:=\frac{b_1+\ldots+b_N}{\sqrt N}\label{eq-SN}\end{equation} should converge in some reasonable sense, nor that the limit should turn out to be a natural refinement of the Wick formula for classical Gaussians, but that indeed turns out to be the case. The following theorem is the ``almost sure" version of the Central Limit Theorem of Speicher, presented as an amalgamation of Theorem~1 of \cite{Speicher1992} and Lemma~1 of \cite{Speicher1992}. Throughout this paper, $\mathscr P_2(2n)$ will denote the collection of pair-partitions of $[2n]$, with each $\mathscr V\in\mathscr P_2(2n)$ uniquely written as $\mathscr V=\{(w_1,z_1),\ldots,(w_n,z_n)\}$ for $w_1<\ldots<w_n$ and $w_i<z_i$ $(i=1,\ldots,n)$. For further prerequisite definitions, the reader is referred to Section~\ref{preliminaries}.

\begin{assump}
Given a $\ast$-algebra $\mathscr A$ and a state $\varphi:\mathscr A\to\mathbb C$, consider a sequence $\{b_i\}_{i\in\mathbb N}$ of elements of $\mathscr A$ satisfying the following:
\begin{enumerate}
\item for all $i\in\mathbb N$, $\varphi(b_i)=\varphi(b_i^\ast)=0$;
\item for all for all $i,j\in\mathbb N$ with $i<j$ and $\epsilon,\epsilon'\in\{1,\ast\}$, $\varphi(b_i^{\epsilon}b_i^{\epsilon'}) = \varphi(b_j^{\epsilon}b_j^{\epsilon'})$; 
\item for all $n\in\mathbb N$ and all $j(1),\ldots,j(n)\in\mathbb N$, $\epsilon(1),\ldots,\epsilon(n)\in\{1,\ast\}$, the corresponding mixed moment is uniformly bounded, viz. $|\varphi(\prod_{i=1}^nb_{j(i)}^{\epsilon(i)})|\leq \alpha_n$ for some non-negative real $\alpha_n$;
\item $\varphi$ factors over the naturally ordered products in $\{b_i\}_{i\in\mathbb N}$, in the sense of Definition~\ref{natordprod}.
\end{enumerate}
Assume that for all $i\neq j$ and all $\epsilon,\epsilon'\in\{1,\ast\}$, $b_i^{\epsilon}$ and $b_j^{\epsilon'}$ satisfy the commutation relation
\begin{equation}b_i^{\epsilon}b_j^{\epsilon'}\,\,=\,\,s(j,i)\,b_j^{\epsilon'}b_i^{\epsilon},\quad\quad\quad s(j,i)\in \{-1,1\}.\label{comrelassumpSp}\end{equation}\label{assump-CLT-Sp}
\end{assump}

\begin{thm}[Non-commutative CLT \cite{Speicher1992}]
Consider a non-commutative probability space $(\A,\varphi)$ and a sequence of elements $\{b_i\}_{i\in\mathbb N}$ in $\A$ satisfying Condition~\ref{assump-CLT-Sp}. Fixing $q\in [-1,1]$, let the commutation signs $\{s(i,j)\}_{1\leq i< j}$ be drawn from the collection of independent, identically distributed random variables taking values in $\{-1,1\}$ with $\E(\mathbf{s}(i,j))=q$. Then, for almost every sign sequence  $\{s(i,j)\}_{1\leq i< j}$, the following holds: for every $n\in\mathbb N$ and all $\epsilon(1),\ldots,\epsilon(2n)\in \{1,\ast\}$,
\begin{align}&\lim_{N\to\infty}\varphi(S_N^{\epsilon(1)}\ldots S_N^{\epsilon(2n-1)})=0,&\label{eqLimit1S}\\
&\lim_{N\to\infty}\varphi(S_N^{\epsilon(1)}\ldots S_N^{\epsilon(2n)})=\sum_{\mathscr V\in \mathscr P_2(2n)}q^{\cross(\mathscr V)}\prod_{i=1}^n\varphi(b^{\epsilon(w_i)}b^{\epsilon(z_i)})\,,&\label{eqLimit2S}\end{align}
with $S_N\in\A$ as given in \eqref{eq-SN}, $\mathscr V=\{(w_1,z_1),\ldots,(w_n,z_n)\}$, and where $\cross(\mathscr V)$ denotes the number of crossings in $\mathscr V$ (cf. Definition~\ref{defCrossNest}).\label{thm-CLT-Sp}
\end{thm}

The moment expressions in (\ref{eqLimit1S}) and (\ref{eqLimit2S}) may seem familiar. Indeed, this stochastic setting with the average value of the commutation coefficient set to $q$ turns out to be, from the point of view of limiting distributions, equivalent to the setting of bounded linear operators on the the \emph{$q$-Fock space} $\F_q(\H)$ of Bo\.zejko and Speicher \cite{Bozejko1991}. Specifically, given a real, separable Hilbert space $\H$ and two elements $f,g\in\H$, the \emph{creation} and \emph{annihilation} operators on $\F_q(\H)$, $a_q(f)^\ast$ and $a_q(g)$ respectively, satisfy the \emph{$q$-commutation relation}:
\begin{equation} a_q(f)a_q(g)^\ast-q\,a_q(g)^\ast a_q(f)=\langle f,g\rangle_\H\, 1.\label{qcomm}\end{equation}
The mixed moments with respect to the \emph{vacuum expectation state} $\varphi_q$ of these operators are given by a Wick-type formula which, compared against (\ref{eqLimit1S}) and (\ref{eqLimit2S}), yields that for a unit vector $e$ in $\H$, 
$\lim_{N\to\infty}\varphi(S_{N}^{\epsilon(1)}\ldots S_{N}^{\epsilon(n)})=\varphi_q(a_q(e)^{\epsilon(1)}\ldots a_q(e)^{\epsilon(n)})$ for all $n\in\mathbb N$ and  $\epsilon(1),\ldots,\epsilon(2n)\in\{1,\ast\}$. As described in \cite{Speicher1992}, Theorem~\ref{thm-CLT-Sp} can be used to provide a general asymptotic model for operators realizing the relation \eqref{qcomm}, thus providing non-constructive means of settling the conjecture in \cite{Frisch1970}.

Finally, any sequence $\{b_i\}_{i\in [n]}$ satisfying Condition~\ref{assump-CLT-Sp} has a $\ast$-representation on $\mathscr A_n:=\mathcal M_2(\mathbb R)^{\otimes n}$, where $\mathcal M_2(\mathbb R)$ denotes the algebra of $2\times 2$ real matrices. Matricial models for operators satisfying the canonical anti-commutation relation, i.e. the fermionic case corresponding to $q=-1$ in \eqref{qcomm}, are well known and are provided by the so-called Jordan-Wigner transform (see e.g. \cite{Carlen1993} for its appearance in a closely-related context). By extending the transform to the setting where there are both commuting and anti-commuting elements and by applying Theorem~\ref{thm-CLT-Sp} , Biane \cite{Biane1997b} obtained a random matrix model for operators satisfying the $q$-commutation relation \eqref{qcomm}. By replacing $2\times 2$ matrices with $4\times 4$ block-diagonal matrices, Kemp \cite{Kemp2005} similarly obtained models for the corresponding complex family $(a(f)+ia(g))/\sqrt{2}$. To describe the extended Jordan-Wigner model, we make the identification $\mathscr A_n\cong \mathcal M_{2^n}(\mathbb R)$ and let the $\ast$ operation be the conjugate transpose on $\mathcal M_{2^n}(\mathbb R)$. Furthermore, let $\varphi_n:\mathcal M_{2^n}(\mathbb R)\to\mathbb C$ be the positive map $a\mapsto \langle a e_0, e_0\rangle_n$, where $\langle,\rangle_n$ is the usual inner product on $\R^n$ and $e_0=(1,0,\ldots,0)$ an element of the standard basis. 

\begin{lem}[Extended Jordan-Wigner Transform \cite{Biane1997b}]
Fix $q\in[-1,1]$ and consider a sequence of commutation coefficients $\{s(i,j)\}_{i\leq j}$ 
drawn from $\{-1,1\}$. Consider the $2\times 2$ matrices $\{\sigma_x\}_{x\in\R}, \gamma$ given as
$$\sigma_x=\left[\begin{array}{cc}1&0\\0& x\end{array}\right],\quad \quad\gamma=\left[\begin{array}{cc}0&1\\0&0\end{array}\right]$$ and, for $i=1,\ldots,n$, let the element $b_{n,i}\in\mathcal M_2(\mathbb C)^{\otimes n}$ be given by
\begin{equation}b_{n,i}=\sigma_{s(1,i)}\otimes\sigma_{s(2,i)}\otimes\ldots\otimes \sigma_{s(i-1,i)}\otimes \gamma\otimes \underbrace{\sigma_{1}\otimes \ldots\otimes\sigma_{1}}_{=\sigma_1^{\otimes (n-i)}}.\label{eq-JWB}\end{equation} 
Then, for every $n\in\mathbb N$, the non-commutative probablity space $(\A_n,\varphi_n)$ and the elements $b_{n,1},b_{n,2},$ $\ldots, b_{n,n}\in\A_n$ satisfy Condition~\ref{assump-CLT-Sp}.\label{lem-JW-Bia}
\end{lem}

\subsection{Main Results}
This article derives a general form of the Non-commutative Central Limit Theorem of \cite{Speicher1992}. The setting now concerns a sequence  $\{b_i\}_{i\in\mathbb N}$ of non-commutative random variables satisfying the commutation relation
\begin{equation}b_i^{\epsilon}b_j^{\epsilon'}\,\,=\,\,\mu_{\epsilon',\epsilon}(j,i)\,b_j^{\epsilon'}b_i^{\epsilon}\quad\quad\text{with }\,\,\epsilon,\epsilon'\in\{1,\ast\},\,\mu_{\epsilon',\epsilon}(i,j)\in \mathbb R,\label{comrelB}\end{equation}
for $i\neq j$. The consistency of the above commutation relation is ensured by requiring that for all $i<j$ and $\epsilon,\epsilon'\in\{1,\ast\}$
\begin{align}&\mu_{1,1}(i,j)=\frac{1}{\mu_{\ast,\ast}(i,j)},&\mu_{1,\ast}(i,j)=\frac{1}{\mu_{\ast,1}(i,j)},\label{prescription1}&\\ 
&\mu_{\ast,1}(i,j)=t\,\mu_{\ast,\ast}(i,j),&\mu_{\epsilon',\epsilon}(j,i)=\frac{1}{\mu_{\epsilon,\epsilon'}(i,j)},\label{prescription2}&\end{align}
where $t>0$ is a fixed parameter that will appear explicitly in the limits of interest. 
The reader is referred to the beginning of Section~\ref{sec-CLT1} and the Remark~\ref{relationRemark} of Section~\ref{sec-CLT2} for a discussion of the relations \eqref{prescription1}--\eqref{prescription2}. The conditions underlying the extended non-commutative CLT are now the following.

\begin{assump}
Given a $\ast$-algebra $\mathscr A$ and a state $\varphi:\mathscr A\to\mathbb C$, consider a sequence $\{b_i\}_{i\in\mathbb N}$ of elements of $\mathscr A$ satisfying the following:
\begin{enumerate}
\item for all $i\in\mathbb N$, $\varphi(b_i)=\varphi(b_i^\ast)=\varphi(b_ib_i)=\varphi(b_i^\ast b_i^\ast)=\varphi(b_i^\ast b_i)=0$;
\item for all $i,j\in\mathbb N$ and $\epsilon,\epsilon'\in\{1,\ast\}$, $\varphi(b_i^{\epsilon}b_i^{\epsilon'}) = \varphi(b_j^{\epsilon}b_j^{\epsilon'})$; 
\item for all $n\in\mathbb N$ and all $j(1),\ldots,j(n)\in\mathbb N$, $\epsilon(1),\ldots,\epsilon(n)\in\{1,\ast\}$, the corresponding mixed moment is uniformly bounded, i.e. $|\varphi(\prod_{i=1}^nb_{j(i)}^{\epsilon(i)})|\leq \alpha_n$ for some non-negative real $\alpha_n$;
\item $\varphi$ factors over the naturally ordered products in $\{b_i\}_{i\in\mathbb N}$, in the sense of Definition~\ref{natordprod}.
\end{enumerate}

Assume that for all $i\neq j$ and all $\epsilon,\epsilon'\in\{1,\ast\}$, $b_i^{\epsilon(1)}$ and $b_j^{\epsilon(2)}$ satisfy the commutation relation
\begin{equation}b_i^{\epsilon}b_j^{\epsilon'}\,\,=\,\,\mu_{\epsilon',\epsilon}(j,i)\,b_j^{\epsilon'}b_i^{\epsilon},\quad\quad\quad\mu_{\epsilon',\epsilon}(j,i)\in\mathbb R\,.\label{comrelassump}\end{equation}
\label{assump-CLT}
\end{assump}

\begin{thm}[Extended Non-commutative CLT]
Consider a noncommutative probability space $(\A,\varphi)$ and a sequence of elements $\{b_i\}_{i\in\mathbb N}$ in $\A$ satisfying Condition~\ref{assump-CLT}. Fix $q\in\mathbb R$, $t>0$ and let $\{\mu(i,j)\}_{1\leq i< j}$ be drawn from a collection of independent, identically distributed, non-vanishing random variables, with
\begin{equation}\E(\mu(i,j))=qt^{-1}\in\mathbb R,\quad\quad \E(\mu(i,j)^2)=1.\label{eq-prob}\end{equation}
Letting $\mu_{\ast,\ast}(i,j)=\mu(i,j)$ for $1\leq i<j$, populate the remaining $\mu_{\epsilon,\epsilon'}(i,j)$, for $\epsilon,\epsilon'\in\{1,\ast\}$ and $i\neq j$ ($i,j\in\mathbb N$), by \eqref{prescription1} and \eqref{prescription2}.

Then, for almost every sequence  $\{\mu(i,j)\}_{i\leq j}$, the following holds: for every $n\in\mathbb N$ and all $\epsilon(1),\ldots,\epsilon(2n)\in \{1,\ast\}$,
\begin{align}&\lim_{N\to\infty}\varphi(S_N^{\epsilon(1)}\ldots S_N^{\epsilon(2n-1)})=0,\label{eqLimit1B}&\\
&\lim_{N\to\infty}\varphi(S_N^{\epsilon(1)}\ldots S_N^{\epsilon(2n)})=\sum_{\mathscr V\in \mathscr P_2(2n)}q^{\cross(\mathscr V)}t^{\nest(\mathscr V)}\prod_{i=1}^n\varphi(b^{\epsilon(w_i)}b^{\epsilon(z_i)})\label{eqLimit2B}&\end{align}
with $S_N\in\A$ as given in \eqref{eq-SN}, $\mathscr V=\{(w_1,z_1),\ldots,(w_n,z_n)\}$, and where $\cross(\mathscr V)$ denotes the number of crossings in $\mathscr V$ and $\nest(\mathscr V)$ the number of nestings in $\mathscr V$ (cf. Definition~\ref{defCrossNest}). \label{CLT1}
\end{thm}

The generalized commutation structure of Theorem~\ref{CLT1}, compared to  Theorem~\ref{thm-CLT-Sp}, generates a second combinatorial statistic in the limiting moments -- that of \emph{nestings} in pair partitions. This refinement also extends to the Fock-space level, with the above limits realized as moments of creation and annihilation operators on the \emph{$(q,t)$-Fock space} \cite{Blitvic1}, briefly overviewed in Section~\ref{preliminaries}. Compared to \eqref{qcomm}, the operators of interest now satisfy the commutation relation
\begin{equation}a_{q,t}(f)a_{q,t}(g)^\ast -q \,a_{q,t}(g)^\ast a_{q,t}(f)= \langle f,g\rangle_{_\H}\, t^{N},\label{qtcomm}\end{equation}
where $N$ is the \emph{number operator}. Note that \eqref{eq-prob}, together with the fact that $t>0$, recovers the fundamental constraint that $|q|<t$ in order for the $(q,t)$-Fock space to be a bona fide Hilbert space.

The extended non-commutative CLT of Theorem~\ref{CLT1} requires that $\varphi(b_ib_i)=\varphi(b_i^\ast b_i^\ast)=\varphi(b_i^\ast b_i)=0$ (cf. Condition~\ref{assump-CLT}), not generally needed in Speicher's setting (Condition~\ref{assump-CLT-Sp}) except in the asymptotic models for the $q$-commutation relation \eqref{qcomm}. Remark~\ref{orderingRemark} of Section~\ref{sec-CLT1} and Remark~\ref{limitRemark} of Section~\ref{sec-CLT2} discuss the existence of the limits \eqref{eqLimit1B} and \eqref{eqLimit2B} and the form they take in the absence of this requirement. Nevertheless, this additional condition is in fact consistent with the natural choice of matrix models -- specifically, extending Lemma~\ref{lem-JW-Bia}, the following is a generalized Jordan-Wigner construction. The underlying probability space $(\A_n,\varphi_n)$ remains that of the previous section.

\begin{lem}[Two-parameter Jordan-Wigner Transform]
Fix $q\in\mathbb R$, $t>0$ and let $\{\mu_{\epsilon,\epsilon'}(i,j)\}_{i\neq j,\epsilon,\epsilon'\in\{1,\ast\}}$ be a sequence of commutation coefficients, i.e. a sequence of non-zero real numbers satisfying \eqref{prescription1} and \eqref{prescription2}. Consider the $2\times 2$ matrices $\{\sigma_x\}_{x\in\mathbb R},\gamma$ given by 
$$\sigma_{x}=\left[\begin{array}{cc}1&0\\0&\sqrt{t}\,x\end{array}\right],\quad \quad\gamma=\left[\begin{array}{cc}0&1\\0&0\end{array}\right].$$ For $i=1,\ldots,n$, let $\mu(i,j):=\mu_{\ast,\ast}(i,j)$ and consider the element $b_{n,i}\in\mathcal M_2(\mathbb R)^{\otimes n}$ given by
\begin{equation}b_{n,i}=\sigma_{\mu(1,i)}\otimes \sigma_{\mu(2,i)}\otimes\ldots\otimes\sigma_{\mu(i-1,i)}\otimes \gamma\otimes\underbrace{\sigma_{1}\otimes\ldots\otimes\sigma_{1}}_{=\sigma_1^{\otimes (n-i)}}.\label{eq-JW-B}\end{equation} 
Then, for every $n\in\mathbb N$, the non-commutative probablity space $(\A_n,\varphi_n)$ and the elements $b_{n,1},b_{n,2},$ $\ldots, b_{n,n}\in\A_n$ satisfy Condition~\ref{assump-CLT}.
\label{thm-JW-2}
\end{lem}

Finally, analogously to \cite{Biane1997b}, Theorem~\ref{CLT1} and Lemma~\ref{thm-JW-2} together yield an asymptotic random matrix models for the creation, annihilation, and field operators on the $(q,t)$-Fock space: 
\begin{cor}
Consider a sequence of commutation coefficients drawn according to Theorem~\ref{CLT1} and the corresponding matrix construction of Lemma~\ref{thm-JW-2}. Let
\begin{equation}S_{N,k}:=\frac{1}{\sqrt N}\sum_{i=N(k-1)+1}^{Nk}b_{Nk,i}.\end{equation}
Then, for any choice of $k,i(1),\ldots,i(k)\in\mathbb N$, $\epsilon(1),\ldots,\epsilon(k)\in\{1,\ast\}$,
\begin{align}&\lim_{N\to\infty}\varphi_{Nk}(S_{N,i(1)}^{\epsilon(1)}\ldots S_{N,i(k)}^{\epsilon(k)})=\varphi_{q,t}(a_{q,t}(e_1)^{\epsilon(1)}\ldots a_{q,t}(e_{k})^{\epsilon(k)})\end{align}\label{thm-asymptoticB}
and
\begin{align}&\lim_{N\to\infty}\varphi_{Nk}((S_{N,i(1)}+S_{N,i(1)}^\ast)\ldots (S_{N,i(k)}+S_{N,i(k)}^\ast))=\varphi_{q,t}(s_{q,t}(e_1)\ldots s_{q,t}(e_{k}))\end{align}\label{thm-asymptoticB}
where $\varphi_{q,t}$ is the vacuum expectation state on the $(q,t)$-Fock space $\F_{q,t}(\H)$, operator $a_{q,t}(e_i)$ is the twisted annihilation operator on $\F_{q,t}(\H)$ associated with the element $e_i$ of the orthonormal basis of $\H$, and $s_{q,t}(e_i):=a_{q,t}(e_i)+a_{q,t}(e_i)^\ast$ is the corresponding field operator.\label{cor1}
\end{cor}

\label{introduction}

\section{Preliminaries}
\label{preliminaries}
The present section overviews the key combinatorial constructs used to encode the mechanics of the non-commutative Central Limit Theorem. It also overviews the Hilbert space framework that provides a natural setting in which to realize the limits of the random matrix models of Corollary~\ref{cor1}.

\subsection{Partitions}

Denote by $\mathscr P(n)$ the collection of partitions of $[n]:=\{1,\ldots,n\}$. Set partitions will be extensively used to encode equivalence classes of products of random variables, based on the repetition patterns of individual elements. Specifically, any two $r$-vectors will be declared equivalent if element repetitions occur at same locations in both vectors; i.e.
for $(i(1),\ldots,i(r)),(j(1),\ldots,j(r))\in [N]^r$,
\begin{eqnarray}(i(1),\ldots,i(r))\sim (j(1),\ldots,j(r))&\iff&\text{ for all } 1\leq k_1<k_2\leq r,\nonumber\\
&&i(k_1)=i(k_2)\text{ iff }j(k_1)=j(k_2).\label{eqSim}\end{eqnarray}
It then immediately follows that the equivalence classes of ``$\sim$'' can be identified with the set $\mathscr P(r)$ of the partitions of $[r]$. An example is shown in Figure~\ref{example-equiv}. Note that writing ``$(i(1),\ldots,i(r))\sim\mathscr V$'' will indicate that $(i(1),\ldots,i(r))$ is in the equivalence class identified with the partition $\mathscr V\in\mathscr P(r)$.
\begin{figure}[h]\centering
\includegraphics[scale=0.5]{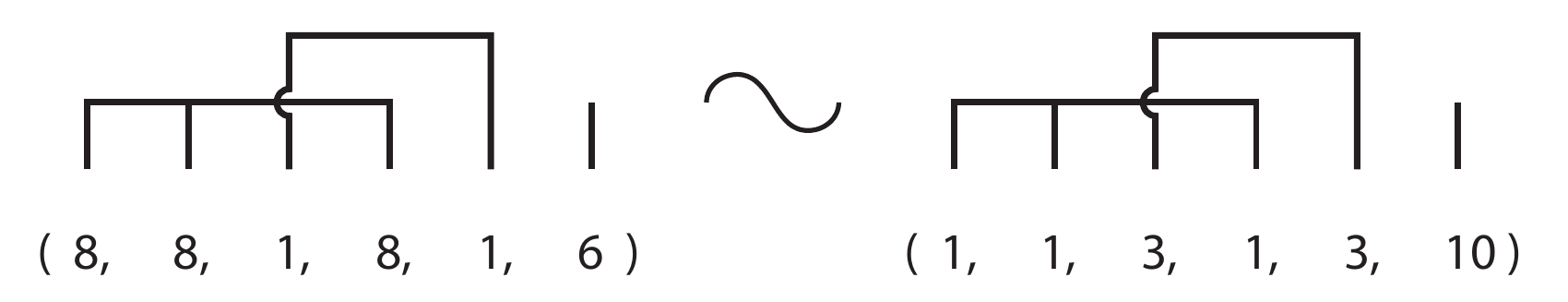}
\caption{Two elements of $[N]^r$ (for $N=10,r=6$) that belong to the same equivalence class, where the latter is represented as the corresponding partition $\mathscr V\in \mathscr P(r)$ given by $\mathscr V=\{(1,2,4),(3,5),(6)\}$.}\label{example-equiv}
\end{figure}

Particularly relevant is the collection $\mathscr P_2(2n)$ of \emph{pair partitions} of $[2n]$, also referred to as \emph{pairings} or \emph{perfect matchings}, which are partitions whose each part contains exactly two elements. A pair partition will be represented as a list of ordered pairs, that is, $\mathscr P_2(2n)\ni\mathscr V=\{(w_1,z_1),\ldots,(w_n,z_n)\}$, where $w_i<z_i$ for $i\in[n]$ and $w_1<\ldots<w_n$. In the present setting, the pair partitions will typically appear with additional refinements given by the following two statistics on $\mathscr P_2(2n)$.

\begin{defn}[Crossings and Nestings] For $\mathscr V=\{(w_1,z_1),\ldots,(w_{n},z_{n})\}\in\mathscr P_2(2n)$, pairs $(w_i,z_i)$ and $(w_j,z_j)$ are said to \emph{cross} if $w_i<w_j<z_i<z_j$. The corresponding \emph{crossing} is encoded by $(w_i,w_j,z_i,z_j)$ with 
$\Cross(\mathscr V):=\{(w_i,w_j,z_i,z_j)\mid (w_i,z_i),(w_j,z_j)\in\mathscr V\text{ with } w_i<w_j<z_i<z_j\}$ as the set of all crossings in $\mathscr V$ and $\cross(\mathscr V):=|\Cross(\mathscr V)|$ counting the crossings in $\mathscr V$.

For $\mathscr V=\{(w_1,z_1),\ldots,(w_{n},z_{n})\}\in\mathscr P_2(2n)$, pairs $(w_i,z_i)$ and $(w_j,z_j)$ are said to \emph{nest} if $w_i<w_j<z_j<z_i$. The corresponding \emph{nesting} is encoded by $(w_i,w_j,z_j,z_i)$ with
$\Nest(\mathscr V):=\{(w_i,w_j,z_j,z_i)\mid (w_i,z_i),(w_j,z_j)\in\mathscr V\text{ with } w_i<w_j<z_j<z_i\}$ as the set of all nestings in $\mathscr V$ and $\nest(\mathscr V):=|\Nest(\mathscr V)|$ counting the nestings in $\mathscr V$.
\label{defCrossNest}
\end{defn}
The two concepts are illustrated in Figures~\ref{crnest1} and \ref{crnest2}, by visualizing the pair partitions as collections of disjoint chords with end-points labeled (increasing from left to right) by elements in $[2n]$. 

\begin{figure}\centering\includegraphics[scale=0.5]{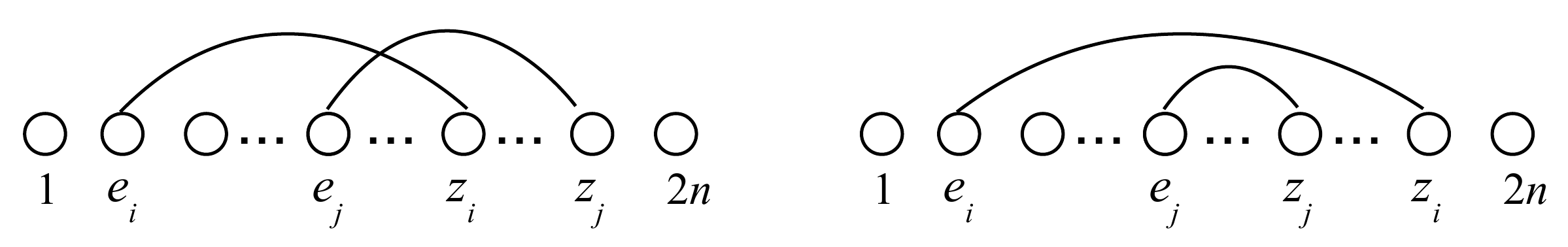}
\caption{An example of a crossing [left] and nesting [right] of a pair partition $\mathscr V=\{(e_1,z_1),\ldots,(e_{n},z_{n})\}$ of $[2n]$.}
\label{crnest1}
\end{figure}
\begin{figure}\centering\includegraphics[scale=0.5]{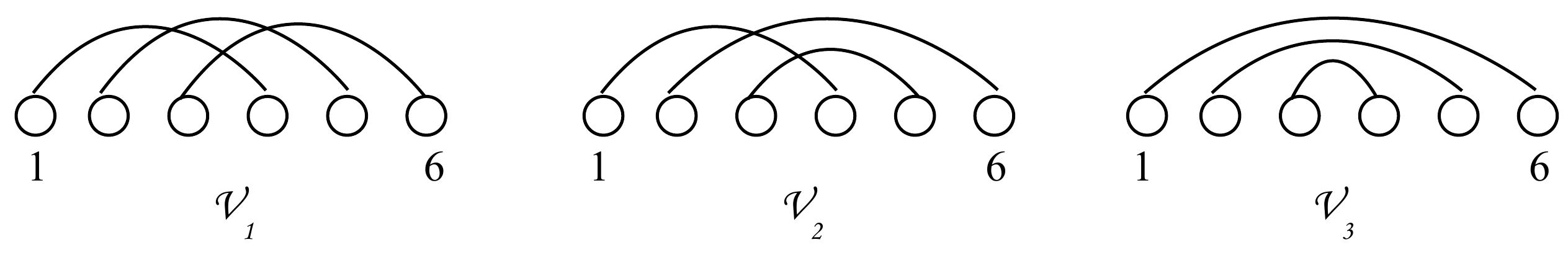}
\caption{Example of three pair partitions on $[2n]=\{1,\ldots,6\}$: $\cross(\mathscr V_1)=3,\,\nest(\mathscr V_1)=0$ [left], $\cross(\mathscr V_2)=2,\,\nest(\mathscr V_2)=1$ [middle], $\cross(\mathscr V_3)=0,\,\nest(\mathscr V_3)=3$ [right].}
\label{crnest2}
\end{figure}

\subsection{Operators on the $(q,t)$-Fock space}
\label{qt_Fock}

The $(q,t)$-Fock space $\F_{q,t}(\H)$ \cite{Blitvic1}, for $|q|<t$, is a two-parameter deformation of the classical Bosonic and Fermionic Fock spaces. Consider the tensor algebra on the Hilbert space $\H$ (taken as real and separable) given by
$\F(\H):=\bigoplus_{n\geq 0} (\mathbb C\otimes \H)^{\otimes n},$
with $(\mathbb C\otimes \H)^0$ defined as a complex vector space spanned by a real unit vector $\Omega\not\in\mathscr H$. The algebra $\F(\H)$ is spanned by the pure tensors $\{h_1\otimes\ldots\otimes h_n\mid n\in\mathbb N, h_1,\ldots,h_n\in\H\}\cup\{\Omega\}$.
The completion of $\F(\H)$ with respect to the usual inner product, denoted $\langle\,,\,\rangle_0$ and given by $\langle \Omega,\Omega\rangle_0$ and $\langle f_1\ldots f_n, h_1\ldots h_m\rangle_0=\delta_{n,m}\langle f_1, h_1\rangle_\H\ldots \langle f_n, h_n\rangle_\H$, yields the \emph{full (Boltzmann) Fock space}. In the present scenario, it will be more interesting to complete with respect to the ``$(q,t)$-symmetrized" inner product $\langle\,,\,\rangle_{q,t}$ given by $\langle \Omega,\Omega\rangle_{q,t}=1$ and 
\begin{eqnarray}
\langle f_1\otimes \ldots \otimes f_n,h_1\otimes\ldots\otimes h_m\rangle_{q,t}=\left\{\begin{array}{ll}0,&n\neq m\\
t^{n\choose 2}\sum_{\pi\in S_n}(qt^{-1})^{\inv(\pi)}\langle f_1,h_{\pi(1)}\rangle_\H\ldots \langle f_n,h_{\pi(n)}\rangle_\H,&n=m
\end{array}\right.
\end{eqnarray}
where $\inv(\pi)$ denotes the number of inversions of the permutation $\pi\in S_n$. The completion of $\F(\H)$ with respect to $\langle\,,\,\rangle_{q,t}$ yields the \emph{$(q,t)$-Fock space}  $\F_{q,t}(\H)$, where letting $t=1$ recovers the $q$-Fock space $\F_q(\H)$ of Bo\.zejko and Speicher \cite{Bozejko1991}.

The annihilation operators $\{a_{q,t}(h)\}_{h\in\H}$ on $\F_{q,t}(\H)$  and their adjoints (with respect to $\langle\,,\,\rangle_{q,t}$), the creation operators $\{a_{q,t}(h)^\ast\}_{h\in\H}$, are densely defined on $\F(\H)$ by
\begin{eqnarray}a_{q,t}(f)^\ast \Omega = f,&& a_{q,t}(f)^\ast h_1\otimes\ldots\otimes h_n=f\otimes h_1\otimes\ldots\otimes h_n\label{cdef1}\\
a_{q,t}(f)\Omega = 0, && a_{q,t}(f)h_1\otimes\ldots\otimes h_n=\sum_{k=1}^n q^{k-1}t^{n-k} \langle f,h_k\rangle_{_\H}\,  h_1\otimes \ldots\otimes\breve h_k\otimes\ldots\otimes h_n,\label{cdef2}\end{eqnarray}
where the superscript $\breve h_k$ indicates that $h_k$ has been deleted from the product. Letting $t^N$ be the linear operator defined by $t^N\Omega=\Omega$ and $t^N h_1\otimes\ldots\otimes h_n=t^n h_1\otimes\ldots\otimes h_n$, the creation and annihilation operators are readily shown to satisfy the $(q,t)$-commutation relation \eqref{qtcomm}. The two-parameter family of the (self-adjoint) field operators $s_{q,t}(h):=a_{q,t}(h)+a_{q,t}(h)^\ast$, for $h\in\H$, is referred to as a $(q,t)$-Gaussian family.
The \emph{moments} of the creation, annihilation, and field operators are computed with respect to the vacuum expectation state $\varphi_{q,t}:\mathscr B(\F_{q,t}(\H))\to\mathbb C$, $\varphi_{q,t}(a)=\langle\, a\,\Omega,\Omega\rangle_{q,t}$. In particular, for every $n\in\mathbb N$ and all $\epsilon(1),\ldots,\epsilon(2n)\in \{1,\ast\}$,
\begin{align}\displaystyle&\varphi_{q,t}(a_{q,t}(h_1)^{\epsilon(1)}\ldots a_{q,t}(h_{2n-1})^{\epsilon(2n-1)})=0,&\\
&\varphi_{q,t}(a_{q,t}(h_1)^{\epsilon(1)}\ldots a_{q,t}(h_{2n})^{\epsilon(2n)})=\sum_{\mathscr V\in \mathscr P_{2}(2n)}q^{\cross(\mathscr V)}t^{\nest(\mathscr V)}\prod_{i=1}^n\varphi(a_{q,t}(h_{w_i})^{\epsilon(w_i)}a_{q,t}(h_{z_i})^{\epsilon(z_i)}),&\\
&\varphi_{q,t}(s_{q,t}(h)^{2n-1})=0,&\\
&\varphi_{q,t}(s_{q,t}(h)^{2n})=\|h\|_\H^{2n} \sum_{\mathscr V\in \mathscr P_2(2n)}  q^{\cross(\mathscr V)}\,t^{\nest(\mathscr V)},&
\end{align}
where $\mathscr P_{2}(2n)$ is again the collection of pair partitions of $[2n]$ and each $\mathscr V\in \mathscr P_{2}(2n)$ is (uniquely) written as a collection of pairs $\{(w_1,z_1),\ldots,(w_n,z_n)\}$ with $w_1<\ldots<w_n$ and $w_i<z_i$.

\section{The Extended Non-commutative Central Limit Theorem}
\label{sec-CLT1}
The goal of this section is to extend the ``deterministic formulation'' of the non-commutative Central Limit Theorem of Speicher \cite{Speicher1992}. The deterministic result differs from the previously stated Theorem~\ref{thm-CLT-Sp} in that the sequence of commutation signs $(s(i,j))_{i,j}$,  taking values in $\{-1,1\}$ and associated with the commutation relations  $b_i^\epsilon b_j^{\epsilon'}=s(j,i) b_j^{\epsilon'}b_i^\epsilon$, is now \emph{fixed}. In \cite{Speicher1992}, an analogous Wick-type formula is nevertheless shown to exist, provided the existence of the following limit:
$$\lim_{N\to\infty}\,\,\,\frac{1}{N^{2n}}\!\! \sum_{\substack{i(1),\ldots,i(2n)\in[N]\text{ s.t. }\\(i(1),\ldots,i(2n))\sim\mathscr V}}\prod_{(w_j,w_k,z_j,z_k)\in \text{Cross}(\mathscr V)} s(i(w_j),i(w_k))\quad:=\lambda_{\mathscr V}$$
for each pair partition $\mathscr V\in\mathscr P(2n)$.

At present, the focus is on a sequence  $(b_i)_{i\in\mathbb N}$ of non-commutative random variables satisfying a more general type of commutation relations, where for all $i\neq j$ and $\epsilon,\epsilon'\in\{1,\ast\}$,
\begin{equation}b_i^{\epsilon}b_j^{\epsilon'}\,\,=\,\,\mu_{\epsilon',\epsilon}(j,i)\,b_j^{\epsilon'}b_i^{\epsilon}\quad\quad\text{ for some }\,\,\mu_{\epsilon',\epsilon}(i,j)\in \mathbb R.\label{comrelB}\end{equation}
At the outset, the sequence of commutation coefficients $\{\mu_{\epsilon,\epsilon'}(i,j)\}_{i\neq j,\epsilon,\epsilon'\in\{1,\ast\}}$ 
must satisfy certain properties. In particular, interchanging the roles of $i$ and $j$ in the commutation relation implies that
\begin{equation}\mu_{\epsilon,\epsilon'}(i,j)=\frac{1}{\mu_{\epsilon',\epsilon}(j,i)}.\tag{A}\label{tagA}\end{equation}
Similarly, conjugating (via the $\ast$ operator) both sides of the commutation relation yields
$$\mu_{\ast,\ast}(i,j)=\mu_{1,1}(j,i),\quad \mu_{1,\ast}(i,j)=\mu_{1,\ast}(j,i),\quad \mu_{\ast,1}(i,j)=\mu_{\ast,1}(j,i).$$
(For example, $b_ib_j=\mu_{1,1}(j,i)b_jb_i$ and therefore $b_j^\ast b_i^\ast=(b_ib_j)^\ast=\mu_{1,1}(j,i)b_i^\ast b_j^\ast$, but also $b_j^\ast b_i^\ast=\mu_{\ast,\ast}(i,j) b_i^\ast b_j^\ast$.) Therefore, by (A),
\begin{equation*}\mu_{\ast,\ast}(i,j)=\frac{1}{\mu_{1,1}(i,j)},\quad\quad \mu_{\ast,1}(i,j)=\frac{1}{\mu_{1,\ast}(i,j)}.\tag{B}\end{equation*}
 
The second key ingredient in a non-commutative CLT is a moment-factoring assumption. As in \cite{Speicher1992}, the factoring is assumed to follow the underlying partition structure. Drawing on the notation of Section~\ref{preliminaries}, viz. the equivalence relation ``$\sim$'' on the set $[N]^r$ of $r$-tuples in $[N]:=\{1,\ldots,N\}$, the two relevant ways in which the moments may be assumed to factor are defined as follows. 

\begin{defn} Consider a sequence $\{b_i\}_{i\in\mathbb N}$ of random variables, elements of some non-commutative probability space $(\A,\varphi)$. The element $b_{i(1)}^{\epsilon(1)}\ldots b_{i(n)}^{\epsilon(n)}$, for $\epsilon(1),\ldots,\epsilon(n)\in\{1,\ast\}$ and $i(1),\ldots,i(n)\in\mathbb N$, is said to be an \emph{interval-ordered product} if $(i(1),\ldots,i(n))\sim \mathscr V$ where $\mathscr V=\{\{1,\ldots,k_1\},\{k_1+1,\ldots,k_2\},\ldots,\{k_{|\mathscr V|-1}+1\ldots,k_{|\mathscr V|}\}\}$ is an interval partition of $[n]$. The same element is said to be a \emph{naturally ordered product} if, in addition, $i(1)< i(k_1+1)\ldots< i(k_{|\mathscr V|-1}+1)$.

The state $\varphi$ is said to \emph{factor over naturally (resp. interval) ordered products} in $\{b_i\}_{i\in\mathbb N}$ if
$$\varphi(b_{i(1)}^{\epsilon(1)}\ldots b_{i(n)}^{\epsilon(n)})=\varphi(b_{i(1)}^{\epsilon(1)}\ldots b_{i(k_1)}^{\epsilon(k_i)})\ldots \varphi(b_{i(k_{|\mathscr V|-1}+1)}^{\epsilon(k_{|\mathscr V|-1}+1)}\ldots b_{i(k_{|\mathscr V|})}^{\epsilon(k_{|\mathscr V|})})$$
whenever $b_{i(1)}^{\epsilon(1)}\ldots b_{i(n)}^{\epsilon(n)}$ is a naturally (resp. interval) ordered product.
\label{natordprod}
\end{defn}

The following remark ensures that the commutation relations \eqref{comrelB} are consistent with the moment factoring assumptions.

\begin{remark} In assuming $\varphi$ factors over naturally ordered products, one must be able to bring a moment $\varphi(b_i^{\epsilon_i}b_i^{\epsilon_i'}b_j^{\epsilon_j}b_j^{\epsilon_j'})$ for $i>j$ into naturally-ordered form. Alternatively, should it be further assumed that $\varphi$ factors over interval-ordered products of the sequence $\{b_i\}_{i\in\mathbb N}$, one must allow that $\varphi(b_i^{\epsilon_i}b_i^{\epsilon_i'}b_j^{\epsilon_j}b_j^{\epsilon_j'})=\varphi(b_j^{\epsilon_j}b_j^{\epsilon_j'}b_i^{\epsilon_i}b_i^{\epsilon_i'})$ for all $i,j$ and $\epsilon,\epsilon'\in\{1,\ast\}$. When commutation coefficients are constrained to take values in $\{-1,1\}$, it is in fact the case that $b_i^{\epsilon_i}b_i^{\epsilon_i'}$ commutes with $b_j^{\epsilon_j}b_j^{\epsilon_j'}$, and the moment-factoring assumptions are consistent with the commutativity structure. However, this need not be the case for the general setting. In particular, 
$$\varphi(b_i^{\epsilon_i}b_i^{\epsilon_i'}b_j^{\epsilon_j}b_j^{\epsilon_j'})=\mu_{\epsilon_i,\epsilon_j'}(j,i)\mu_{\epsilon_i,\epsilon_j}(j,i)\mu_{\epsilon_i',\epsilon_j'}(j,i)\mu_{\epsilon_i',\epsilon_j}(j,i)\,\varphi(b_j^{\epsilon_j}b_j^{\epsilon_j'}b_i^{\epsilon_i}b_i^{\epsilon_i'}).$$
The reader may verify that any sequence of real-valued commutation coefficients for which the above product evaluates to unity  regardless of the choice of $\epsilon,\epsilon'$ must in fact take values in $\{-1,1\}$. 

Instead, rather than imposing additional restrictions on the sign sequence, the alternative approach is that of restricting the range of $\varphi$ when applied to the sequence $\{b_i\}$. In particular, by (A)--(B), 
$$\mu_{\epsilon_i,\epsilon_j'}(j,i)\mu_{\epsilon_i,\epsilon_j}(j,i)\mu_{\epsilon_i',\epsilon_j'}(j,i)\mu_{\epsilon_i',\epsilon_j}(j,i)=1$$
whenever $\epsilon\neq \epsilon'$. Thus, by imposing that $\varphi(b_i^\ast b_i^\ast)=\varphi(b_i b_i)=0$ for all $i\in\mathbb N$, the assumption on the factoring of naturally-ordered second moments conveniently becomes equivalent to factoring of interval-ordered second moments. Note that factoring an interval-ordered product containing higher moments generally still incurs a product of commutation coefficients. However, as will become apparent shortly, the contribution of such expressions vanishes in the limits of interest.\label{orderingRemark}
\end{remark}

The stage is now set for the main result of this section.

\begin{thm}[Extended Non-commutative CLT]
Consider a non-commutative probability space $(\A,\varphi)$ and a sequence $\{b_i\}_{i\in\mathbb N}$ of elements of $\A$ satisfying Condition~\ref{assump-CLT}, with the real-valued commutation coefficients $\{\mu_{\epsilon',\epsilon}(i,j)\}$ satisfying the consistency conditions (A)--(B). For $n\in\mathbb N$, fix $\epsilon(1),\ldots,\epsilon(2n)\in\{1,\ast\}$ and, letting $\mathscr P_2(2n)$ denote the collection of pair partitions of $[2n]$, assume that for all $\mathscr V=\{(w_1,z_1),\ldots,(w_n,z_n)\}\in \mathscr P_2(2n)$ the following limit exists:
\begin{eqnarray}\lambda_{\mathscr V,\epsilon(1),\ldots,\epsilon(2n)}&:=&\lim_{N\to\infty}\,\,\,N^{-n}\!\! \sum_{\substack{i(1),\ldots,i(2n)\in[N]\text{ s.t. }\\(i(1),\ldots,i(2n))\sim\mathscr V}}\left(\prod_{\substack{(w_j,w_k,z_j,z_k)\\\in\, \text{Cross}(\mathscr V)}} \mu_{\epsilon(z_j),\epsilon(w_k)}(i(z_j),i(w_k))\right.\nonumber\\&&\hspace{2cm} \left.\prod_{\substack{(w_j,w_m,z_m,z_j)\\\in\, \text{Nest}(\mathscr V)}} \mu_{\epsilon(z_j),\epsilon(z_m)}(i(z_j),i(z_m))\mu_{\epsilon(z_j),\epsilon(w_m)}(i(z_j),i(w_m))\right)\,,\label{eqLimit}\end{eqnarray}
where $\text{Cross}(\mathscr V)$ and $\text{Nest}(\mathscr V)$ denote, respectively, the sets of crossings and nestings in $\mathscr V$ (cf. Definition~\ref{defCrossNest}) and where the equivalence relation $\sim$ is given by (\ref{eqSim}).

\vspace{10pt}
Then, for every $n\in\mathbb N$ and all $\epsilon(1),\ldots,\epsilon(2n)\in \{1,\ast\}$,
\begin{align}&\lim_{N\to\infty}\varphi(S_N^{\epsilon(1)}\ldots S_N^{\epsilon(2n-1)})=0,\label{eqLimit1}&\\
&\lim_{N\to\infty}\varphi(S_N^{\epsilon(1)}\ldots S_N^{\epsilon(2n)})=\sum_{\mathscr V\in \mathscr P_2(2n)}\lambda_{\mathscr V,\epsilon(1),\ldots,\epsilon(2n)}\prod_{i=1}^n\varphi(b^{\epsilon(w_i)}b^{\epsilon(z_i)}),\label{eqLimit2}&\end{align}
for $S_N\in\A$ as given in \eqref{eq-SN} and with each $\mathscr V\in \mathscr P_2(2n)$ written as  $\mathscr V=\{(w_1,z_1),\ldots,(w_n,z_n)\}$ for $w_1<\ldots<w_n$ and $w_i<z_i$ $(i=1,\ldots,n)$.
\label{CLT}
\end{thm}

{\it Proof of Theorem~\ref{CLT}.} The notation and the development follow closely those of \cite{Speicher1992}.

Fix  $r\in\mathbb N$ and $\epsilon(1),\ldots,\epsilon(r)\in\{1,\ast\}$ and recall that the focus is the $N\to\infty$ limit of the corresponding mixed moment of $S_N$. Namely, let
$$M_N:=\varphi(S_N^{\epsilon(1)}\ldots S_N^{\epsilon(r)})=\frac{1}{N^{r/2}}\sum_{i(1),\ldots,i(r)\in[N]}\varphi\left(b_{i(1)}^{\epsilon(1)}\ldots b_{i(r)}^{\epsilon(r)}\right).$$
Making use of the previously-defined equivalence relation, $M_N$ can be rewritten as
$$M_N=\sum_{\mathscr V\in\mathscr P(r)} \frac{1}{N^{r/2}}\sum_{\substack{i(1),\ldots,i(r)\in[N]\text{ s.t. }\\(i(1),\ldots,i(r))\sim\mathscr V}}\varphi\left(b_{i(1)}^{\epsilon(1)}\ldots b_{i(r)}^{\epsilon(r)}\right)=\sum_{\mathscr V\in\mathscr P(r)} \frac{1}{N^{r/2}}\,\,M_N^{\mathscr V},$$
where
$$M_N^{\mathscr V}:=\sum_{\substack{i(1),\ldots,i(r)\in[N]\text{ s.t. }\\(i(1),\ldots,i(r))\sim\mathscr V}}\varphi\left(b_{i(1)}^{\epsilon(1)}\ldots b_{i(r)}^{\epsilon(r)}\right).$$
Focusing on $M_N^{\mathscr V}$, suppose first that $\mathscr V$ contains a singleton, i.e. a single-element part $\{k\}\in\mathscr V$ for some $k\in [r]$. Via the commutation relation (\ref{comrelassump}), $b_{i(1)}^{\epsilon(1)}\ldots b_{i(r)}^{\epsilon(r)}$ can be brought into a naturally ordered form (incurring, in the process, a multiplying factor given by the corresponding product of the commutation coefficients). In turn, by the assumption on the factoring of the naturally ordered products (cf. Definition~\ref{natordprod}), $\varphi\left(b_{i(1)}^{\epsilon(1)}\ldots b_{i(r)}^{\epsilon(r)}\right)$ factors according to the blocks in $\mathscr V$. Since $\varphi (b_k)=\varphi (b_k^\ast)=0$, it follows that for all $N\in\mathbb N$, $M_N^{\mathscr V}=0$ for all partitions $\mathscr V$ containing a singleton block.

Focus next on partitions of $[r]$ containing blocks with two or more elements or, equivalently, partitions $\mathscr V\in \mathscr P(r)$ with $|\mathscr V|\leq \lfloor r/2 \rfloor $, where $|\mathscr V|$ denotes the number of blocks in $\mathscr V$. Recalling that, by the assumption on the existence of uniform bounds on the moments, we have that for all $\mathscr V\in \mathscr P(r)$,
$$\left|\varphi\left(b_{i(1)}^{\epsilon(1)}\ldots b_{i(r)}^{\epsilon(r)}\right)\right|\leq \alpha_{\mathscr V}$$
for some $\alpha_{\mathscr V}\in\mathbb R$. Thus, for a partition $\mathscr V$ with $\ell$ blocks, summing over all $i(1),\ldots,i(r)\in[N]$ with $(i(1),\ldots,i(r))\sim\mathscr V$ yields
$$|M_N^{\mathscr V}|\leq {N\choose \ell}\,\ell!\,\, \alpha_{\mathscr V},$$
and therefore
$$|M_N|\leq \sum_{\mathscr V\in\mathscr P(r)} \frac{{N\choose \ell}\,\ell!}{N^{r/2}}\, \alpha_{\mathscr V}.$$
Noting that (1) the above sum is taken over a fixed (finite) index $r$, (2) that the only $N$-dependent term in the above expression is the ratio ${N\choose \ell}/{N^{r/2}}$ and (3) that  ${N\choose \ell}/{N^{r/2}} \to 0$ as $N\to\infty$ for $\ell <\lceil r/2 \rceil$, it follows that only those partitions $\mathscr V$ with $|\mathscr V|\geq \lceil r/2 \rceil$ contribute to the $N\to\infty$ limit of $M_N$. But, since $|\mathscr V|\leq \lfloor r/2 \rfloor $, it follows that the only non-vanishing contributions are obtained for $r$ even and partitions with exactly $r/2$ blocks | i.e. pair-partitions, $\mathscr V\in \mathscr P_2(r)$. Therefore, for $r$ odd,
$$\lim_{N\to\infty}\varphi(S_N^{\epsilon(1)}\ldots S_N^{\epsilon(r)})=0$$
and, otherwise,
$$\lim_{N\to\infty}|M_N|= \sum_{\mathscr V\in\mathscr P_2(r)} \lim_{N\to\infty}N^{-\frac{r}{2}} \sum_{\substack{i(1),\ldots,i(r)\in[N]\text{ s.t. }\\(i(1),\ldots,i(r))\sim\mathscr V}}\varphi\left(b_{i(1)}^{\epsilon(1)}\ldots b_{i(r)}^{\epsilon(r)}\right).$$
Next, fixing $i(1),\ldots,i(r)\in[N]$ and recalling that $\mathscr V$ is a pair-partition of $[r]$, consider the following algorithm for transforming $b_{i(1)}^{\epsilon(1)}\ldots b_{i(r)}^{\epsilon(r)}$, via the commutation relation (\ref{comrelassump}), into an interval-ordered product. Starting with $i(1)$ and recalling that $\mathscr V$ is a pair-partition of $[r]$, let $1<k_1\leq r$ denote the unique index for which $i(k_1)=i(1)$. Consider element $b_{i(1)}$ to be already in place and commute $b_{i(k_1)}$ with the elements to its left until $b_{i(k_1)}$ is immediately to the right of $b_{i(1)}$, recording all the while the commutation coefficients incurred in each transposition. The next iteration, proceeding in the analogous manner, is carried out on the string of length $r-2$ given by $i(2),\ldots,i(\breve{k_1}),\ldots,i(r)$, where $i(\breve{k_1})$ indicates that $i(k_1)$ has been suppressed from the string. Continuing in this manner, the algorithm terminates when the remaining string is the empty string. The resulting moment is of the form
$$\varphi\left(b_{i(1)}^{\epsilon(1)}\ldots b_{i(r)}^{\epsilon(r)}\right)=\beta^{\epsilon(1),\ldots,\epsilon(r)}_{i(1),\ldots,i(r)}\,\,\varphi\left(b_{i(w_1)}^{\epsilon(w_{1})}b_{i(z_{1})}^{\epsilon(z_{1})}\ldots b_{i(w_{r/2})}^{\epsilon(w_{r/2})}b_{i(z_{r/2})}^{\epsilon(z_{r/2})}\right),$$
where $\mathscr V=\{(w_1,z_1),\ldots,(w_{r/2},z_{r/2})\}$ with $w_1<\ldots<w_{r/2}$ is the underlying pair-partition and $\beta^{\epsilon(1),\ldots,\epsilon(r)}_{i(1),\ldots,i(r)}$ denotes the product of the commutation coefficients incurred in this transformation. Note that though $i(w_j)=i(z_j)$ for all $j=1,\ldots,r/2$, in general $\epsilon(w_j)\neq \epsilon(z_j)$ and the above expression therefore also (artificially) distinguishes between $i(w_j)$ and $i(z_j)$. 

While it need not be the case that $i(w_1)<\ldots<i(w_{r/2})$, and the moment $$\varphi\left(b_{i(w_1)}^{\epsilon(w_{1})}b_{i(z_{1})}^{\epsilon(z_{1})}\ldots b_{i(w_{r/2})}^{\epsilon(w_{r/2})}b_{i(z_{r/2})}^{\epsilon(z_{r/2})}\right)$$ therefore need not be naturally ordered, $\varphi$ nevertheless factors over the pairs. Specifically, as $\varphi(b_jb_j)=\varphi(b_j^\ast b_j^\ast)=0$, it can be assumed that $\epsilon(w_j)\neq \epsilon(z_j)$ for $j=1,\ldots,r/2$. 
By Remark~\ref{orderingRemark}, it then follows that $$\varphi\left(b_{i(1)}^{\epsilon(1)}\ldots b_{i(r)}^{\epsilon(r)}\right)=\beta^{\epsilon(1),\ldots,\epsilon(r)}_{i(1),\ldots,i(r)}\,\,\varphi\left(b_{i(w_1)}^{\epsilon(w_1)}b_{i(z_1)}^{\epsilon(z_1)}\right)\ldots \varphi\left(b_{i(w_{r/2})}^{\epsilon(w_{r/2})}b_{i(z_{r/2})}^{\epsilon(w_{r/2})}\right).$$

Next, $\beta^{\epsilon(1),\ldots,\epsilon(r)}_{i(1),\ldots,i(r)}$ can be expressed combinatorially as follows. Fixing some $j \in [r/2]$ and considering the corresponding pair $(w_j,z_j)\in \mathscr V$ (where $w_j<z_j$), note that for every $k\in [r/2]$ for which $w_j<w_k<z_j<z_k$, the above algorithm commutes $z_j$ and $w_k$. Additionally note that this commutation is performed exactly once, on the $j^\text{th}$ iteration, as the process does not revisit pairs that were brought into the desired form in one of the previous steps. The corresponding contribution to $\beta^{\epsilon(1),\ldots,\epsilon(r)}_{i(1),\ldots,i(r)}$ is therefore given by $\mu_{\epsilon(z_j),\epsilon(w_k)}(i(z_j),i(w_k))$.
Similarly, for every $m\in [r/2]$ for which $w_j<w_m<z_m<z_j$, the above algorithm commutes $z_j$ and $z_m$ \underline{as well as} $z_j$ and $w_m$, and both commutations occur exactly once. The corresponding contribution to $\beta^{\epsilon(1),\ldots,\epsilon(r)}_{i(1),\ldots,i(r)}$ is therefore given by $\mu_{\epsilon(z_j),\epsilon(z_m)}(i(z_j),i(z_m))\mu_{\epsilon(z_j),\epsilon(w_m)}(i(z_j),i(w_m))$. Recall now (cf. Definition~\ref{defCrossNest}) that the $4$-tuple given by $w_j<w_k<z_j<z_k$ is what is referred to as a \emph{crossing} in $\mathscr V$  and encoded by $(w_j,w_k,z_j,z_k)\in\text{Cross}(\mathscr V)$, whereas the $4$-tuple $w_j<w_m<z_m<z_j$ is referred to as a \emph{nesting} in $\mathscr V$ and is encoded as $(w_j,w_m,z_m,z_j)\in \text{Nest}(\mathscr V)$. 
Finally, realizing that the algorithm performs no other commutations than the two types described, it follows that
\begin{eqnarray*}\beta^{\epsilon(1),\ldots,\epsilon(r)}_{i(1),\ldots,i(r)}&=& \prod_{(w_j,w_k,z_j,z_k)\in \text{Cross}(\mathscr V)} \mu_{\epsilon(z_j),\epsilon(w_k)}(i(z_j),i(w_k))\times\\
&&\prod_{(w_j,w_m,z_m,z_j)\in \text{Nest}(\mathscr V)} \mu_{\epsilon(z_j),\epsilon(z_m)}(i(z_j),i(z_m))\mu_{\epsilon(z_j),\epsilon(w_m)}(i(z_j),i(w_m)).\end{eqnarray*}
The encoding of $\beta^{\epsilon(1),\ldots,\epsilon(r)}_{i(1),\ldots,i(r)}$ through nestings and crossings of $\mathscr V$ is illustrated in Figures~\ref{figAlgo1} and \ref{figAlgo2}. 

Putting it all together,
\begin{equation}\lim_{N\to\infty}\varphi(S_N^{\epsilon(1)}\ldots S_N^{\epsilon(2n)})=\sum_{\mathscr V\in \mathscr P_2(2n)}\lim_{N\to\infty}N^{-\frac{r}{2}}\!\!\!\!\!\!\!\!\!\!\sum_{\substack{i(1),\ldots,i(r)\in[N]\text{ s.t. }\\(i(1),\ldots,i(r))\sim\mathscr V}}\!\!\!\!\!\varphi\left(b_{i(w_1)}^{\epsilon(w_1)}b_{i(z_1)}^{\epsilon(z_1)}\right)\ldots \varphi\left(b_{i(w_n)}^{\epsilon(w_n)}b_{i(z_n)}^{\epsilon(z_n)}\right)\beta^{\epsilon(1),\ldots,\epsilon(r)}_{i(1),\ldots,i(r)}\label{eqGeneral}\end{equation}
By the assumption on the covariances of the $b_i$'s and the existence of the limit in (\ref{eqLimit}),
$$\lim_{N\to\infty}\varphi(S_N^{\epsilon(1)}\ldots S_N^{\epsilon(2n)})=\sum_{\mathscr V\in \mathscr P_2(2n)}\varphi\left(b^{\epsilon(w_1)}b^{\epsilon(z_1)}\right)\ldots \varphi\left(b^{\epsilon(w_n)}b^{\epsilon(z_n)}\right)\lim_{N\to\infty}N^{-\frac{r}{2}}\!\!\!\!\!\!\!\!\!\!\sum_{\substack{i(1),\ldots,i(r)\in[N]\text{ s.t. }\\(i(1),\ldots,i(r))\sim\mathscr V}}\!\!\!\!\!\beta^{\epsilon(1),\ldots,\epsilon(r)}_{i(1),\ldots,i(r)},$$
which yields \eqref{eqLimit2} and completes the proof. \hfill$\qed$

\begin{figure}\centering
\includegraphics[scale=0.4]{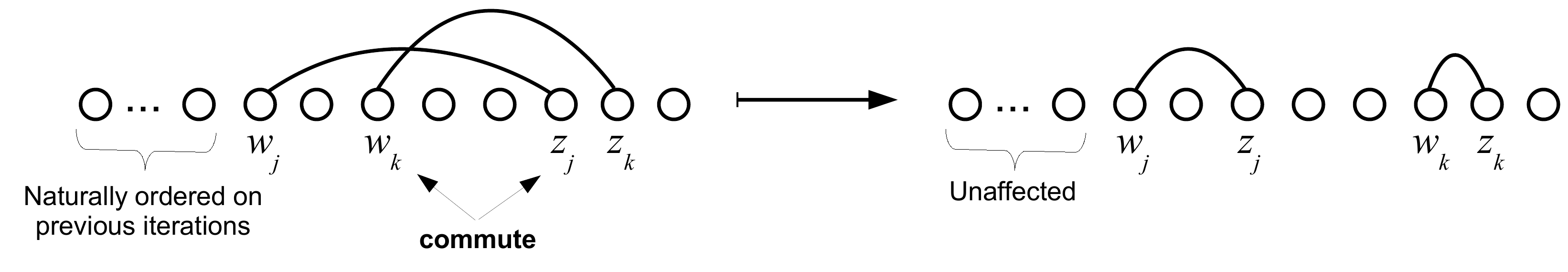}
\caption{The process of bringing a mixed moment into a naturally-ordered form involves
commuting all the inversions and all the nestings in each of the underlying pair partitions. In commuting a crossing $(w_j,w_k,z_j,z_k)$, as depicted, the corresponding moment incurs a factor $\mu_{\epsilon(z_j),\epsilon(w_k)}(i(z_j),i(w_k))$.}
\label{figAlgo1}
\end{figure}

\begin{figure}\centering
\includegraphics[scale=0.4]{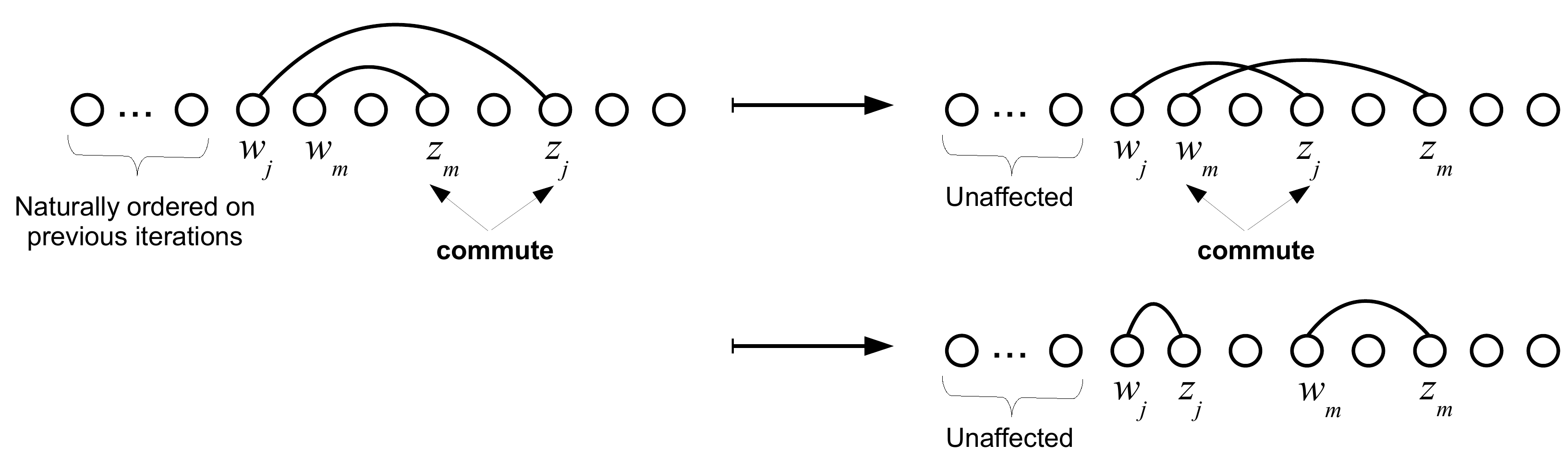}
\caption{The process of bringing a mixed moment into a naturally-ordered form involves
commuting all the inversions and all the nestings in each of the underlying pair partitions. In commuting a nesting $(w_j,w_m,z_m,z_j)$, as depicted, the corresponding moment incurs a factor $\mu_{\epsilon(z_j),\epsilon(z_m)}(i(z_j),i(z_m))\mu_{\epsilon(z_j),\epsilon(w_m)}(i(z_j),i(w_m))$.}
\label{figAlgo2}
\end{figure}

\begin{remark}
The assumption of Theorem~\ref{CLT} that the covariances of the $b_i$'s are independent of $i$, namely, that $\varphi(b_i^{\epsilon_1}b_i^{\epsilon_2}) = \varphi(b_j^{\epsilon_1}b_j^{\epsilon_2})$ for all $i,j\in\mathbb N$ and $\epsilon_1,\epsilon_2\in\{1,\ast\}$, was not used in obtaining (\ref{eqLimit1}) and (\ref{eqGeneral}). Provided the existence of the limit in (\ref{eqGeneral}), the additional assumption is solely used for the purpose of simplifying (\ref{eqGeneral}) as (\ref{eqLimit2}). 
\end{remark}

The above Theorem~\ref{CLT} differs from Theorem~1 of \cite{Speicher1992} in the following ways:
\begin{enumerate}
\item The more general commutation relation $b_i^{\epsilon}b_j^{\epsilon'}=\mu_{\epsilon',\epsilon}(j,i)\,b_j^{\epsilon'}b_i^{\epsilon}$ with $\mu_{\epsilon,\epsilon'}(i,j)\in \mathbb R$ now replaces the commutation relation $b_i^{\epsilon}b_j^{\epsilon'}=s(i,j) b_j^{\epsilon'}b_i^{\epsilon}$ with spins $s(i,j)\in \{-1,1\}$.
\item For the purpose of factoring naturally ordered second moments as interval-ordered second moments, it is presently additionally assumed that $\varphi(b_i^\ast b_i^\ast)=\varphi(b_ib_i)=0$. (cf. Remark~\ref{orderingRemark}.)
\item The convergence of the moments now hinges on the existence of a more complicated limit, which is not only a function of the commutation coefficients and of the underlying partition, as was the case in \cite{Speicher1992}, but also on the pattern of adjoints in the mixed moment of interest (i.e. on the string $\epsilon(1),\ldots,\epsilon(n)$).
\end{enumerate}
Note that the assumption on the uniform bounds on the moments is not new, but is instead implicit in \cite{Speicher1992}.

\section{Stochastic Interpolation}
\label{sec-CLT2}
Recall that, analogously to Theorem~1 in \cite{Speicher1992}, the ``deterministic version'' of the non-commutative CLT hinges on an existence of the limit (\ref{eqLimit}), which is determined by the sequence of commutation coefficients $\{\mu_{\epsilon,\epsilon'}(i,j)\}$. Rather than providing more explicit conditions for the existence of the above limit, this section follows the philosophy of \cite{Speicher1992} and instead considers the scenario where the coefficients ``may have been chosen at random''. The outcome will be that, starting with a probability law for a single coefficient and extending it to a product measure on the entire coefficient sequence, almost any choice of commutation coefficients will yield a finite and easily describable limit. For this, it is first necessary to define a suitable product measure on the coefficient sequence that is consistent with the dependency structure given by (A)--(B), which is accomplished in Remark~\ref{relationRemark}. In turn, Remark~\ref{limitRemark} considers the effect on the limit achieved by imposing the vanishing of certain second moments. Finally, Lemma~\ref{lemRandom} is the remaining ingredient in the ``almost sure version" of the non-commutative CLT (viz. the present Theorem~\ref{CLT1}).

\begin{remark} Defining a measure on the sequence of commutation coefficients by focusing on the triangular sequences $\{\mu_{\ast,\ast}(i,j)\}_{1\leq i< j}$ and attempting to fix the remaining coefficients via (A)--(B) still leaves one degree of freedom. Namely, $\mu_{\ast,\ast}(i,j)$ was not until now explicitly related to $\mu_{\ast,1}(i,j)$. The need for a third relation governing the sign sequence comes into play when considering positivity requirements. Generally, $\varphi$ is assumed to be \emph{positive}, that is, if $\varphi(aa^\ast)\geq 0$ for all $a\in\mathscr A$. Then, $\varphi(b_ib_i^\ast)\geq 0$ and $\varphi(b_ib_jb_j^\ast b_i^\ast)\geq 0$ for all $i,j\in \mathbb N$. But, by the commutation relations and the factoring of naturally ordered moments, 
$$\varphi(b_ib_jb_j^\ast b_i^\ast)= \mu_{\ast,1}(i,j)\mu_{\ast,\ast}(i,j)\varphi(b_ib_i^\ast)\varphi(b_jb_j^\ast).$$
If the sequence $b_1,b_2,\ldots$ is such that $\varphi(b_ib_i^\ast)>0$ for all $i$, the commutation signs must therefore also satisfy the following, third, requirement:
\begin{equation*}\frac{\mu_{\ast,1}(i,j)}{\mid\mu_{\ast,1}(i,j)\mid}=\frac{\mu_{\ast,\ast}(i,j)}{\mid\mu_{\ast,\ast}(i,j)\mid}.\tag{C}\end{equation*}
In the random setting, (C) translates to $\mu_{\ast,1}(i,j)=\gamma(i,j)\mu_{\ast,\ast}(i,j)$ for some random sequence $\{\gamma(i,j)\}$ supported on $(0,\infty)$. 

In assuming $\{\gamma(i,j)\}$ to be independent of $\mu_{\ast,\ast}(i,j)$ in line with the general philosophy of this section, the reader may soon verify that only the expectation of $\gamma(i,j)$ will matter from the perspective of Lemma~\ref{lemRandom}. Furthermore, since the expectations of $\mu_{\ast,\ast}(i,j)$ and $\mu_{\ast,1}(i,j)$ will be taken to not depend on the index $(i,j)$, one is free to set $t:=\E(\gamma(i,j))$. Then, for $i<j$, (C) becomes:
\begin{equation*}\mu_{\ast,1}(i,j)=t\mu_{\ast,\ast}(i,j),\quad t>0.\tag{C'}\end{equation*}\label{relationRemark}
\end{remark}

\begin{remark} Beyond the existence of the limit (\ref{eqLimit}), the goal of the present section is to develop a probabilistic framework in which this limit takes on a particularly natural form. For this purpose, the basic setting of Theorem~\ref{CLT} will need to fulfill an additional requirement. Specifically, by the assumption of factoring of naturally-ordered moments, 
$\varphi(b_ib_jb_i^\ast b_j^\ast)$ and $\varphi(b_ib_j^\ast b_i^\ast b_j)$ for $i<j$ are both brought into their naturally ordered form by performing a single commutation. In the former case, the commutation incurs a factor $\mu_{\ast,1}(i,j)$ and, in the latter, the factor $\mu_{\ast,\ast}(i,j)$. Yet, in the combinatorial formulation, both products are in the equivalence class (in the sense of ``$\sim$") of the pair partition $\pi=$\includegraphics[scale=0.2]{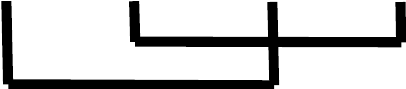} and both are brought into their naturally ordered form by commuting the single crossing in $\pi$.  Thus, in order for the combinatorial invariant to be preserved, either: 
\begin{itemize}
\item the expected values of $\mu_{\ast,1}(i,j)$ and $\mu_{\ast,\ast}(i,j)$ must be the same,\\
\begin{centering}or,\end{centering}
\item one of the two mixed moments vanishes, i.e. $\varphi(b_ib_i^\ast)=0$ or $\varphi(b_i^\ast b_i)=0$ for all $i\in\mathbb N$.
\end{itemize}
By Remark~\ref{relationRemark}, one may without loss of generality let $\mu_{\ast,1}(i,j)=t\,\mu_{\ast,\ast}(i,j)$. Thus, as the reader may soon be able to verify, opting to make equal the means of $\mu_{\ast,1}(i,j)$ and $\mu_{\ast,\ast}(i,j)$ by letting $t=1$ reduces the statistics of the desired limit to those of crossings and the outcome is the same as in the case of randomly chosen commutation signs in \cite{Speicher1992}. The formulation of Lemma~\ref{lemRandom} instead opts for the second alternative, and the introduction of the second parameter $t$ will give rise to the appearance of a second combinatorial statistic, that of \emph{nestings}.

Note that while $\varphi$ is assumed to be positive, it is not assumed to be faithful, and there is no contradiction in assuming that $\varphi(b_i^\ast b_i)=0$ while  $\varphi(b_i b_i^\ast)\neq 0$. As further discussed in the following section, letting $\varphi(b_i^\ast b_i)=0$ and $\varphi(b_i b_i^\ast)=1$ will provide an asymptotic model for a family of ``twisted" annihilation operators, whereas making the opposite choice would yield the corresponding analogue for the creation operators.\label{limitRemark}
\end{remark}

\begin{lem} Fix $0\leq |q|<t$ and let $\{\mu(i,j)\}_{1\leq i< j}$ be a collection of independent, identically distributed non-vanishing random variables, with
$$\E(\mu(i,j))=qt^{-1}\in\mathbb R,\quad\quad \E(\mu(i,j)^2)=1.$$
Letting $\mu_{\ast,\ast}(i,j)=\mu(i,j)$ for $1\leq i<j$, populate the remaining $\mu_{\epsilon,\epsilon'}(i,j)$ for $\epsilon,\epsilon'\in\{1,\ast\}$ by
\begin{align*}&\mu_{1,1}(i,j)=\frac{1}{\mu_{\ast,\ast}(i,j)},&\mu_{1,\ast}(i,j)=\frac{1}{\mu_{\ast,1}(i,j)},&\\ 
&\mu_{\ast,1}(i,j)=t\,\mu_{\ast,\ast}(i,j),&\mu_{\epsilon',\epsilon}(j,i)=\frac{1}{\mu_{\epsilon,\epsilon'}(i,j)}.&\end{align*} 

Then, for any $\mathscr V\in \mathscr P_2(2n)$ and $\epsilon(1),\ldots,\epsilon(2n)\in\{1,\ast\}$, the limit (\ref{eqLimit}) exists a.s. 
Moreover, if $\mathscr V$ is such as to satisfy $(\epsilon(w),\epsilon(z))=(1,\ast)$ for all blocks $(w,z)\in \mathscr V$, the corresponding limit is given by
$$\lambda_{\mathscr V,\epsilon(1),\ldots,\epsilon(2n)}=q^{\cross(\mathscr V)}\,t^{\nest(\mathscr V)}\quad \text{a.s.}$$
where $\cross(\mathscr V)=|\Cross(\mathscr V)|$ and $\nest(\mathscr V)=|\Nest(\mathscr V)|$ denote, respectively, the numbers of crossings and nestings in $\mathscr V$ (cf. Definition~\ref{defCrossNest}).\label{lemRandom}
\end{lem}

\begin{proof} Fix $\mathscr V=\{(w_1,z_1),\ldots,(w_n,z_n)\}$ and, recalling that $i(w_m)=i(z_m)$ for all $m\in [n]$, consider the (classical) random variable $X_N$ given by
\begin{eqnarray}X_N&:=&N^{-n}\sum_{\substack{i(1),\ldots,i(2n)\in[N]\text{ s.t. }\\(i(1),\ldots,i(2n))\sim\mathscr V}}\left(\prod_{(w_j,w_k,z_j,z_k)\in \text{Cross}(\mathscr V)} \mu_{\epsilon(z_j),\epsilon(w_k)}(i(z_j),i(w_k))\right.\nonumber\\&&\left.\times\prod_{(w_\ell,w_m,z_m,z_\ell)\in \text{Nest}(\mathscr V)} \mu_{\epsilon(z_\ell),\epsilon(z_m)}(i(z_\ell),i(z_m))\mu_{\epsilon(z_\ell),\epsilon(w_m)}(i(z_\ell),i(w_m))\right),\label{eq-X}\end{eqnarray}
where the sequence of random variables $\{\mu_{\epsilon,\epsilon'}(i,j)\}_{\epsilon,\epsilon'\in\{1,\ast\},i,j\in\mathbb N,i\neq j}$ is obtained by letting $\mu_{\ast,\ast}(i,j)=\mu(i,j)$ and fixing the remaining coefficients as prescribed by (\ref{prescription1})-(\ref{prescription2}). 
The first goal is to compute $\E(X_N)$. By the independence assumption, since the overall product includes no repeated terms, the expectation factors over the products. It therefore suffices to evaluate $\E(\mu_{\epsilon(z_j),\epsilon(w_k)}(i(w_j),i(w_k)))$ for each crossing $(w_j,w_k,z_j,z_k)$ and $\E( \mu_{\epsilon(z_\ell),\epsilon(z_m)}(i(w_\ell),i(w_m))\mu_{\epsilon(z_\ell),\epsilon(w_m)}(i(w_\ell),i(w_m))$ for each nesting  $(w_\ell,w_m,z_m,z_\ell)$ of a given pair-partition. At the outset, recall that every pair-partition $\mathscr V$ contributing to $X_N$  is such that $(\epsilon(w),\epsilon(z))=(1,\ast)$. Then, starting with the crossings and assuming that $i(w_j)=i(z_j)<i(w_k)=i(z_k)$, the corresponding commutation coefficient in \eqref{eq-X} and its expectation are given as 
\begin{equation}\mu_{\ast,1}(i(z_j),i(w_k))=t\,\mu_{\ast,\ast}(i(z_j),i(w_k))=t\,\mu(i(z_j),i(w_k))\,\,\stackrel{\mathbb E}{\longmapsto}\,\, t(q/t)=q.\end{equation}
whose expectation is $t(q/t)=q$.

\noindent When it is instead the case that $i(w_j)=i(z_j)>i(w_k)=i(z_k)$, it suffices to notice that by (A)--(B),  $\mu_{\ast,1}(i,j)=\mu_{\ast,1}(j,i)$. The same conclusion then holds and each crossing therefore contributes a factor of $q$ on average.
Moving on to nestings, let $(w_\ell,w_m,z_m,z_\ell)$ be a nesting. 
If $i(w_\ell)=z_\ell<i(w_m)=i(z_m)$, the corresponding commutation coefficient in \eqref{eq-X} and its expectation are given as
\begin{equation}\begin{array}{l}\mu_{\ast,\ast}(i(z_\ell),i(z_m))\mu_{\ast,1}(i(z_\ell),i(w_m))=\mu_{\ast,\ast}(i(z_\ell),i(z_m))\,t\,\mu_{\ast,\ast}(i(z_\ell),i(w_m))\\=t(\mu(i(w_\ell),i(w_m)))^2\end{array} \,\,\stackrel{\mathbb E}{\longmapsto}\,\, t.\end{equation}
If on the other hand $i(w_\ell)=z_\ell>i(w_m)=i(z_m)$, by (A)--(B) the commutation coefficient and its expectation become
\begin{equation}\begin{array}{l}\mu_{\ast,\ast}(i(z_\ell),i(z_m))\mu_{\ast,1}(i(z_\ell),i(w_m))=(\mu_{\ast,\ast}(i(z_m),i(z_\ell)))^{-1}\mu_{\ast,1}(i(z_\ell),i(w_m))\\
=(\mu(i(z_m),i(z_\ell)))^{-1}\,t\,\mu(i(z_m),i(z_\ell))=t\end{array} \,\,\stackrel{\mathbb E}{\longmapsto}\,\, t.\end{equation}
Thus, each nesting also contributes a factor of $t$. It follows that $\E(X_N)$ is given by
\begin{equation}\E(X_N)=
N^{-n}\,\sum_{\substack{i(1),\ldots,i(2n)\in[N]\text{ s.t. }\\(i(1),\ldots,i(2n))\sim\mathscr V}}q^{\cross(\mathscr V)}\,t^{\nest(\mathscr V)}= \,q^{\cross(\mathscr V)}\,t^{\nest(\mathscr V)}N^{-n}{N\choose n}\,n!.
\end{equation}
Thus, 
\begin{equation}\lim_{N\to\infty}\E(X_N)=q^{\cross(\mathscr V)}\,t^{\nest(\mathscr V)}.\end{equation}

It now remains to show that $\lim_{N\to\infty} X_N=\lim_{N\to\infty} \E(X_N)$ a.s., that is, that for every $\eta>0$,
$\mathbb P\left(\bigcap_{N\geq 1}\bigcup_{M\geq N}\{|X_M-\E(X_M)|\geq\eta\}\right)=0.$
The calculation is analogous to that in \cite{Speicher1992}. By the subadditivity of $\mathbb P$ and a standard application of Markov inequality,
\begin{equation}\mathbb P\left(\bigcup_{M\geq N}\{|X_M-\E(X_M)|\geq\eta\}\right)\leq \sum_{M\geq N}\mathbb P\left(|X_M-\E(X_M)|\geq\eta\right)\leq \frac{1}{\eta^2}\sum_{M\geq N}\E(|X_M-\E(X_M)|^2).\label{eqChebyshev}\end{equation}
In turn,
$\E(|X_M-\E(X_M)|^2)=\E(X_M^2)-\E(X_M)^2$, with $\E(X_M)^2 = q^{2\,\cross(\mathscr V)}\,t^{2\,\nest(\mathscr V)}$ and 
{\small \begin{eqnarray}
\E(X_M^2)&=&M^{-2n}\sum_{\substack{i(1),\ldots,i(2n)\in[N]\text{ s.t. }\\(i(1),\ldots,i(2n))\sim\mathscr V,\\j(1),\ldots,j(2n)\in[N]\text{ s.t. }\\(j(1),\ldots,j(2n))\sim\mathscr V}}\E\left[\prod_{(w_k,w_\ell)\in \text{Cross}(\mathscr V)} \Big(\mu_{\ast,1}(i(w_k),i(w_\ell))\Big)\,\,\,\, \Big(\mu_{\ast,1}(j(w_k),j(w_\ell))\Big)\right.\nonumber\\
&&\left.\hspace{-1.4cm}\prod_{(w_m,w_m)\in \text{Nest}(\mathscr V)}\Big(\mu_{\ast,\ast}(i(w_m),i(w_n))\,\mu_{\ast,1}(i(w_m),i(w_n))\Big)\Big(\mu_{\ast,\ast}(j(w_m),j(w_n))\,\mu_{\ast,1}(j(w_m),j(w_n))\Big)\right]\,,\label{eq-X2}
\end{eqnarray}}
\noindent where, for convenience of notation, each crossing $(w_k,w_\ell,z_k,z_\ell)$ was abbreviated as $(w_k,w_\ell)$, and similarly for the nestings. 
Now suppose that for two choices of indices and the corresponding sets (not multisets) $\{i(1),\ldots,i(2n)\}$ and $\{j(1),\ldots,j(2n)\}$, there is at most one index in common, i.e. suppose that $\{i(1),\ldots,i(2n)\}\cap\{j(1),\ldots,j(2n)\}\leq 1$. In that case, $(i(k),i(k'))\neq (j(m),j(m'))$ for all  $k,k',m,m'\in [2n]$ with $k\neq k',m\neq m'$. By the independence assumption, the above expectation factors over the product (up to the parenthesized terms) and the contribution of each such $\{i(1),\ldots,i(2n)\},\{j(1),\ldots,j(2n)\}$ is simply $q^{2\cross(\mathscr V)}\,t^{2\nest(\mathscr V)}$. Thus, the choices of indices with $\{i(1),\ldots,i(2n)\}\cap\{j(1),\ldots,j(2n)\}\leq 1$ do not contribute to the variance $\E(|X_M-\E(X_M)|^2)$. It now remains to consider the $\Theta(M^{2n-2})$ remaining terms of the sum \eqref{eq-X2}. 

By the Cauchy-Schwarz inequality, the expectation of the product is bounded, and thus 
$$\E(|X_M-\E(X_M)|^2)\leq M^{-2n}M^{2n-2}C=\frac{C}{M^2},$$
where $C$ does not depend on $M$. Since $\sum_{M\geq 0}M^{-2}$ converges,
$$\lim_{N\to\infty}\sum_{M\geq N}\E(|X_M-\E(X_M)|^2)\to 0,$$
and therefore by (\ref{eqChebyshev}),
$$\mathbb P\left(\bigcap_{N\geq 1}\bigcup_{M\geq N}\{|X_M-\E(X_M)|\geq\eta\}\right)=\lim_{N\to\infty}\mathbb P\left(\bigcup_{M\geq N}\{|X_M-\E(X_M)|\geq\eta\}\right)= 0.$$
This completes the proof.
\end{proof}

\section{Random Matrix Models}
\label{sec-CLT3}
Considering some prescribed sequence $\{\mu_{\epsilon,\epsilon'}(i,j)\}_{\epsilon,\epsilon'\in\{1,\ast\},i,j\in\mathbb N,i\neq j}$ of real-valued commmutation coefficients satisfying (A)--(B), Lemma~\ref{lem-JW-Bia} exhibits a set of elements of a matrix algebra that satisfy the corresponding commutativity structure. The construction is analogous to the one given in Lemma~\ref{lem-JW-Bia} and the latter is in fact stated in a form that renders the present generalization natural.

\begin{proof}[Proof of Lemma~\ref{thm-JW-2}]
To show that $b_j^{\epsilon'}\,=\,\mu_{\epsilon',\epsilon}(j,i)\,b_j^{\epsilon'}b_i^{\epsilon}$, it suffices to consider the definitions in \eqref{eq-JW-B} and commute $2\times 2$ matrices. Specifically, let $i<j$ and, by elementary manipulations on tensor products, note that
\begin{eqnarray}
b_ib_j\!\!&=&\!\!\begin{array}{cccccccccccccccc}(\sigma_{\mu(1,i)}&\otimes&\ldots&\otimes&\sigma_{\mu(i-1,i)}&\otimes& \gamma&\otimes&\sigma_{1}&\otimes&\ldots&\otimes&\sigma_{1}&\otimes&\sigma_1^{\otimes (n-j)}\,)&\times\\
(\sigma_{\mu(1,j)}&\otimes&\ldots&\otimes&\sigma_{\mu(i-1,j)}&\otimes&\sigma_{\mu(i,j)}&\otimes&\sigma_{\mu(i+1,j)}&\otimes&\ldots&\otimes&\gamma&\otimes&\sigma_1^{\otimes (n-j)}\,)&\end{array}\nonumber\\
&=&(\sigma_{\mu(1,i)}\sigma_{\mu(1,j)})\otimes\ldots\otimes(\gamma\sigma_{\mu(i,j)})\otimes (\sigma_1\sigma_{\mu(i+1,j)})\otimes\ldots\otimes (\sigma_{1}\gamma)\otimes(\sigma_{1}\sigma_{1})^{\otimes(n-j)}\label{eq-product}
\end{eqnarray}

\noindent Now note that $\sigma_x\sigma_y=\sigma_y\sigma_x$ for all $x,y\in\R$. Moreover, $\gamma\sigma_x=\sqrt{t}x\,\sigma_x\gamma$. Thus,
$$\gamma \sigma_{\mu(i,j)}=\sqrt{t}\mu(i,j)\sigma_{\mu(i,j)}\gamma\quad\quad\text{and}\quad\quad \sigma_{1}\gamma=(\sqrt t)^{-1}\gamma \sigma_{1},$$
and, therefore,
$$b_ib_j=\frac{\sqrt{t}\mu(i,j)}{\sqrt{t}}b_jb_i=\mu_{\ast,\ast}(i,j)b_jb_i=\mu_{1,1}(j,i)b_jb_i.$$
Next, in commuting $b_i^\ast$ with $b_j$, the only non-trivial commutations are that of $\gamma^\ast$ with $\sigma_{\mu(i,j)}$ and $\sigma_{1}^\ast=\sigma_{1}$ with $\gamma$. Since $\gamma^\ast \sigma_x=(\sqrt{t}x)^{-1}\,\sigma_x\gamma^\ast$, it follows that
$$b_i^\ast b_j=\frac{1}{\sqrt t}\,\frac{1}{\sqrt t \mu(i,j)}b_jb_i^\ast=\frac{1}{t\,\mu(i,j)}b_jb_i^\ast=\frac{1}{t\,\mu_{\ast,\ast}(i,j)}b_jb_i^\ast=\frac{1}{\mu_{\ast,1}(i,j)}b_jb_i^\ast=\mu_{1,\ast}(j,i)b_jb_i^\ast.$$
The remaining relations now follow by taking adjoints, and the result is that $b_i^{\epsilon}b_j^{\epsilon'}\,\,=\,\,\mu_{\epsilon',\epsilon}(j,i)\,b_j^{\epsilon'}b_i^{\epsilon}$.

\vspace{10pt}
It remains to show that, in addition to the commutation relation, the resulting matrix sequences also satisfy the assumptions (1)-(4) of Theorem~\ref{CLT1}. Start by noting that for $a_1,\ldots,a_k\in\mathscr M_2$, $\varphi(a_1\otimes\ldots\otimes a_k)=(a_1)_{11}\ldots (a_k)_{11}$, where $(a)_{11}:=\langle e_1 a,e_1\rangle_2$. It therefore immediately follows that for all $i\in\mathbb N$, $\varphi(b_i)=\varphi(b_i^\ast)=0$. By the same token, it is also clear that for all $i,j\in\mathbb N$, $\varphi(b_ib_i^\ast) = \varphi(b_jb_j^\ast)=1$ and $\varphi(b_i^\epsilon b_i^{\epsilon'}) = \varphi(b_i^\epsilon b_i^{\epsilon'})=0$ for $\epsilon,\epsilon'\in\{1,\ast\}$ with $(\epsilon,\epsilon')\neq (1,\ast)$, and, furthermore, $|\varphi(\prod_{i=1}^nb_{j(i)}^{\epsilon(i)})|\leq 1$ for all $n$ and all choices of exponents and indices. The factoring over naturally ordered products also follows immediately, completing the proof.
\end{proof}

Finally, combining Theorem~\ref{CLT1} with Lemma~\ref{thm-JW-2}, and comparing the resulting moments with those given in Section~\ref{qt_Fock}, immediately yields the desired asymptotic models for the creation, annihilation, and field operators on the $(q,t)$-Fock space. For instance, the mixed moments of $S_N$ converge to those of the annihilation operator $a(e_1)$, where $e_1$ is an element of the orthonormal basis of $\H$. More generally, the expressions of Section~\ref{qt_Fock} in fact consider \emph{systems} of operators, e.g. they specify the \emph{joint mixed moments} of annihilation operators $a(e_1),\ldots,a(e_n)$ associated with basis elements $e_1,\ldots,e_n$. In order to asymptotically realize the joint moments of $a(e_1),\ldots,a(e_n)$ rather than the moments of $a(e_1)$ alone, it suffices to consider a sequence $S_{N,1},\ldots, S_{N,n}$ of partial sums built from non-intersecting subsets of $\{b_i\}_{i\in\mathbb N}$. For instance, the fact that $e_i$ and $e_j$ are orthogonal for $i\neq j$ and that the moment $\varphi_{q,t}(a(e_i)a(e_j))$ vanishes follows (in this asymptotic setting) from the fact that $\varphi(b_ib_j)=0$ for $i\neq j$. The general formulation is found in Corollary~1.

\vspace{10pt}

{\small {\bf Acknowledgements} 
\quad The author would like to thank her thesis advisors: Todd Kemp for the conversation that gave rise to the idea of revisiting the non-commutative Central Limit Theorem and his encouragement to pursue that idea further, and Philippe Biane for his advice and for his hospitality at the Institut Gaspard Monge during the final year of the author's Ph.D. The author also gratefully acknowledges the encouragement and advice received from Roland Speicher.
}

\bibliographystyle{abbrv}
\bibliography{qtFock}

\begin{thebibliography}{10}

\bibitem{Biane1997b}
P.~Biane.
\newblock Free hypercontractivity.
\newblock {\em Comm. Math. Phys.}, 184(2):457--474, 1997.

\bibitem{Biane2003}
P.~Biane.
\newblock Free probability for probabilists.
\newblock In {\em Quantum probability communications, {V}ol. {XI} ({G}renoble,
  1998)}, QP-PQ, XI, pages 55--71. World Sci. Publ., River Edge, NJ, 2003.

\bibitem{Blitvic1}
N.~Blitvi\'c.
\newblock The $(q,t)$-gaussian process.
\newblock \emph{Submitted, arxiv:1111.6565}, 2012.

\bibitem{Bozejko1991}
M.~Bo{\.z}ejko and R.~Speicher.
\newblock An example of a generalized {B}rownian motion.
\newblock {\em Comm. Math. Phys.}, 137(3):519--531, 1991.

\bibitem{Carlen1993}
E.~A. Carlen and E.~H. Lieb.
\newblock Optimal hypercontractivity for {F}ermi fields and related
  noncommutative integration inequalities.
\newblock {\em Comm. Math. Phys.}, 155(1):27--46, 1993.

\bibitem{Frisch1970}
U.~Frisch and R.~Bourret.
\newblock Parastochastics.
\newblock {\em J. Mathematical Phys.}, 11:364--390, 1970.

\bibitem{Kemp2005}
T.~Kemp.
\newblock Hypercontractivity in non-commutative holomorphic spaces.
\newblock {\em Communications in Mathematical Physics}, 259:615--637, 2005.

\bibitem{NicaSpeicher}
A.~Nica and R.~Speicher.
\newblock {\em Lectures on the combinatorics of free probability}, volume 335
  of {\em London Mathematical Society Lecture Note Series}.
\newblock Cambridge University Press, Cambridge, 2006.

\bibitem{Speicher1992}
R.~Speicher.
\newblock A noncommutative central limit theorem.
\newblock {\em Math. Z.}, 209(1):55--66, 1992.

\bibitem{Voiculescu1986}
D.~Voiculescu.
\newblock Addition of certain noncommuting random variables.
\newblock {\em J. Funct. Anal.}, 66(3):323--346, 1986.

\bibitem{Voiculescu1992}
D.~V. Voiculescu, K.~J. Dykema, and A.~Nica.
\newblock {\em Free random variables}, volume~1 of {\em CRM Monograph Series}.
\newblock American Mathematical Society, Providence, RI, 1992.

\end{thebibliography}

\end{document}